\documentclass{au} 
\usepackage{amssymb,latexsym}
\usepackage{amsmath}
\usepackage[dvipdfm]{graphicx}
\usepackage{enumerate}
\newenvironment{enumeratei}{\begin{enumerate}[\upshape (i)]}%
                            {\end{enumerate}}
\usepackage{color}

\newcommand \margnote [1] {}

\numberwithin{equation}{section}
\theoremstyle{plain}
 \newtheorem{theorem}{Theorem}[section]
 \newtheorem{lemma}[theorem]{Lemma}
 
 \newtheorem{proposition}[theorem]{Proposition}
 \newtheorem{corollary}[theorem]{Corollary}
\theoremstyle{definition}
  
  \newtheorem{remark}[theorem]{Remark}
  
  \newtheorem{example}[theorem]{Example}
  \newtheorem*{ackno}{Acknowledgment}

\newcommand \tbf [1]{\textbf{#1}}
\renewcommand \emptyset {\varnothing}
\renewcommand \phi {\varphi}
\renewcommand \rho {\varrho}

\newcommand{\id}{{\rm id}}
\renewcommand{\ker}{{\rm ker}\,}
\renewcommand{\min}{{\rm min}}
\renewcommand{\max}{{\rm max}}

\newcommand\meet{\wedge}

\newcommand{\var}[1]{{\mathcal #1}}
\newcommand{\alg}[1]{{\bf #1}}

\newcommand \conj {\mathrel{\&}}
\newcommand \set[1] {\{#1\}}
\newcommand \then {\mathrel{\Rightarrow}}
\newcommand \opposit [1] {{#1}^{\ast}}
\newcommand \FF {G} 
\newcommand \ciklgroupletter {{^{\textup{gr}}\kern-1ptC}}
\newcommand \cikl [1] {\ciklgroupletter_{#1} } 
\newcommand \algcikl [1] {\alg{\ciklgroupletter}_{#1} } 
\newcommand \ciklhat [2] { \ciklgroupletter_{#1}^{#2} } 
\newcommand \twm [2] { #2 \times_{\kern -1pt \textup{tw}}  #1}
\newcommand \gtwm [3] {\mathbf N(#1,\alg #2,\alg #3)} 

\newcommand \pair [2] {\langle #1,#2\rangle}
\newcommand\filter[1]{\mathord\uparrow #1}
\newcommand \gensub [1] {[#1]} 
\newcommand \genpsub [2] {[#2]_{#1}} 
\newcommand \freeagsign  {F}
\newcommand \freeaset [2] {\freeagsign_{#1}(#2)}
\newcommand \freealg [2] {\alg\freeagsign_{#1}(#2)}
\newcommand\mbetu {M}
\newcommand\qalset [2]{\mbetu(#1,#2)} 
\newcommand\qalgeb [2]{\alg \mbetu(#1,#2)} 
\newcommand \mmalg {\qalgeb{\sgrp}{\alg\FF}} 
\newcommand \mmset {\qalset{\sgrp}{\alg\FF}} 
\newcommand \rep {T} 
\newcommand\sgrp {H} 
\newcommand \csillegy {\mathrel{\mathord=^\ast} }
\newcommand \Sub [1] {\textup{Sub}(#1)}
\newcommand \minqvar [1] {\textup{MinQVar}(#1)}
\newcommand \minqvan [1] {\textup{MinQVar}_0(#1)}
\newcommand \minqvae [1] {\textup{MinQVar}_1(#1)}
\newcommand \minqvak [1] {\textup{MinQVar}_2(#1)}
\newcommand \restrict [2] {{#1}\kern-1pt \rceil_{\kern-1pt #2}}
\newcommand \auxset {D(\alg\FF)}

\newcommand \slat[1] { \alg{#1}\textup{-SLat} }

\begin{document} 

\title{Minimal quasivarieties of semilattices over commutative groups}
\author[I.\ V.\ Nagy]{Ildik\'o V.\ Nagy}
\address{Szeged, Hungary, ildi.ildiko.nagy@gmail.com}

\subjclass[2010]{Primary {06A12}, 
secondary {08A35}}

\keywords{{Semilattice over a group, group extension of semilattices, minimal quasivariety}}

\begin{abstract} We  continue some recent investigations of  W.\ Dziobiak,  J.\ Je\v zek, and M.\ Mar\'oti. Let $\alg \FF=\langle \FF ,\cdot\rangle $ be a commutative group. 
A \emph{semilattice over $\alg\FF$}  is a semilattice enriched with $\FF $ as a set of unary operations acting as semilattice automorphisms.  We prove that the  minimal quasivarieties of semilattices over a finite abelian group $\alg \FF$ are in one-to-one correspondence with the subgroups of $\alg \FF$. If $\alg\FF$ is not finite, then we reduce the description of minimal quasivarieties to that of those minimal quasivarieties in which 
not every algebra has a zero element.
\end{abstract}
%

\maketitle

\section{Introduction}\label{sec:intro}
Let $\tau$ be  a homomorphism from an abelian group $\alg \FF=\langle \FF ;\cdot,\id\rangle$ to the automorphism group  of a semilattice $\langle A;\wedge\rangle$. Then the elements $g\in\FF$ become unary operations $g(x)=\tau(g)(x)$ on $A$, and the algebra $\alg A=\langle A;\meet,\FF  \rangle$ obtained this way is a \emph{$\alg \FF $-semilattice}, also called a \emph{semilattice over $\alg \FF$}. Semilattices over abelian groups or their term equivalent variants were investigated in several papers, to be mentioned soon. In particular, W.\ Dziobiak,  J.\ Je\v zek, and M.\ Mar\'oti~\cite{DJM} described the minimal quasivarieties of semilattices over the infinite cyclic group. 
Our goal is to extend their result to other abelian groups. 

Clearly, each minimal quasivariety $\var R$ of semilattices over an abelian group $\alg\FF$ is determined by its $1$-generated free algebra $\freealg{\var R}1$, provided $\freealg{\var R}1$ is nontrivial. Our first  result, Theorem~\ref{t:eq}, describes these minimal $\var R$ by characterizing the free algebras $\freealg{\var R}1$. If $\alg\FF$ is a finite abelian group, then Theorem~\ref{fIndSrpt} gives a much more explicit description by establishing a bijective correspondence between the minimal quasivarieties of $\alg\FF$-semilattices and the subgroups of $\alg \FF$. For infinite abelian groups, only a less explicit description of the minimal quasivarieties $\var R$ is given in Theorem~\ref{rdThemt} since  Theorem~\ref{t:eq} in itself is insufficient for a complete understanding of those $\freealg{\var R}1$  that have no  zero element. 

\subsection{Outline} Section~\ref{secPrelim} gives the basic concepts and notation, including some earlier results. Section~\ref{secmorabout} is devoted to easy statements on $\alg\FF$-semilattices. The first result is stated and proved in Section~\ref{secMain}. Section~\ref{seCtwoconstr} gives two constructions that yield minimal quasivarieties. Minimal quasivarieties of semilattices over finite abelian groups are completed described in Section~\ref{seCfinigrp}, while Section~\ref{secNonFini} is devoted to the infinite case. Finally, Section~\ref{sectiExmpl} points out why the infinite case is much subtler than the finite one.

\section{Preliminaries}\label{secPrelim}
\subsection{Basic concepts and notation}
For a second look at the key concept, an algebra $\alg A=\langle A;\meet,\FF  \rangle$ is called 
a \emph{$\alg \FF $-semilattice}, or a \emph{semilattice over $\alg \FF$},  if $\alg \FF =\langle \FF ;\cdot,\id\rangle$ is a commutative  group, 
the elements of $\FF$ are unary operations acting on the set $A$, and the following identities hold: 
\begin{enumeratei}
\item\label{Fsldefa} $\wedge$ is an associative, idempotent, and commutative operation;
\item\label{Fsldefb} $\id(x)\approx x$;
\item\label{Fsldefc} $f(g(x))\approx (f\cdot g)(x)$ for every $f,g\in \FF $; 
\item\label{Fsldefd} $g(x)\wedge g(y)\approx g(x\wedge y)$ for every $g\in \FF $.
\end{enumeratei}
This definition is due to M.~Mar\'oti~\cite{MM}. 
Axioms \eqref{Fsldefa}--\eqref{Fsldefd} imply that, for every $f\in \FF $,  the map $x\mapsto f(x)$ is an automorphism of the semilattice reduct $\langle A;\meet\rangle$.  
A $\alg \FF $-semilattice  $\alg A$ is \emph{trivial} if it is a singleton. 
If $g(x)=x$ holds for all $x\in A$ and $g\in \FF $, then $\alg A=\langle A;\wedge, \FF \rangle$ is called \emph{$\alg \FF $-trivial.}

Following J.~Je\v zek~\cite{JJ}, an algebra $\langle A;\wedge, g,g^{-1} \rangle$ is a \emph{semilattice with an automorphism} if $\langle A;\wedge \rangle$ is a semilattice, and the unary operations $g$ and $g^{-1}$ are reciprocal automorphisms  of $\langle A;\wedge \rangle$. If  $\alg \FF$ happens to be a cyclic group generated by $g$, then the $\alg\FF$-semilattice $\alg A=\langle A;\wedge, \FF \rangle$ is term equivalent to 
the semilattice  $\langle A;\wedge, g,g^{-1} \rangle$ with an automorphism, that is, these two algebras have the same term functions. For $n\in\mathbb N\cup\set\infty=\set{1,2,\dots,\infty}$, the $n$-element cyclic group is denoted by $\algcikl n$; in this context,  $g$ always stands for a generating element of $\algcikl n$. Notice that a $\algcikl n$-semilattice $\langle A;\wedge, \cikl n \rangle$  is uniquely determined by (but, in lack of $g^{-1}$, not necessarily term equivalent to) its reduct $\langle A;\wedge, g\rangle$;  we often rely on this fact implicitly.

A \emph{quasi-identity} (also called \emph{Horn formula}) is a universally quantified sentence of the form
$(p_1\approx q_1 \conj \cdots \conj  p_n\approx q_n) \then  p\approx q,$
where  $n\in\mathbb N_0=\set{0,1,2,\ldots}$ and $p_1,q_1,\ldots, p_n,q_n,p,q$ are terms.
\emph{Quasivarieties} are classes of (similar) algebras defined by quasi-identities. The least quasivariety  and the least variety containing a given algebra $\alg A$ are denoted by $\var Q(\alg A)$ and $\var V(\alg A)$, respectively. A quasivariety is  \emph{trivial} if it consists of trivial algebras. A nontrivial quasivariety is \emph{minimal} if it has exactly one proper subquasivariety, the trivial one. For concepts and notation not defined in the paper, the  reader is referred to S.\ Burris and H.\,P.\ Sankappanavar~\cite{rBurSan}.

\subsection{Earlier results motivating the present investigations} 
The systematic study of semilattices with an automorphism started in 
J.~Je\v zek~\cite{JJ}, where the simple ones and the subdirectly irreducible ones were described. The simple semilattices with two commuting automorphisms, which can also be considered $({\algcikl\infty}\times{\algcikl\infty})$-semilattices (up to term equivalence), were described in J.~Je\v zek~\cite{J}. Generalizing this result, M.~Mar\'oti~\cite{MM} characterized the simple $\alg\FF$-semilattices for every abelian group~$\FF$. 
\begin{figure}
\includegraphics[scale=1.0]{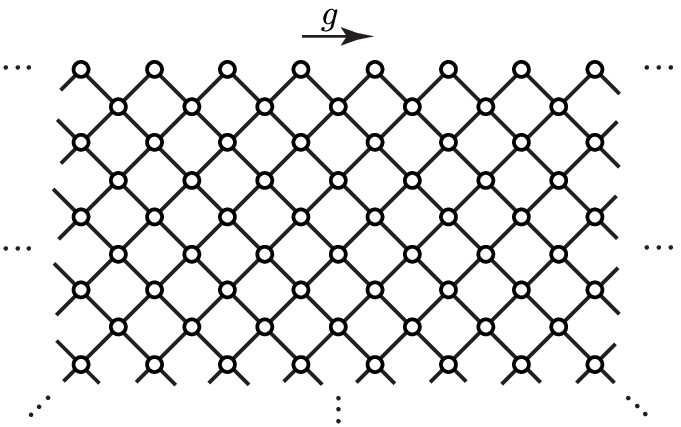}
\caption{$\alg C_1$} 
\label{figone}
\end{figure}
The minimal quasivarieties of $\alg{\cikl\infty}$-semilattices were described by W.~Dziobiak,~J.~Je\v zek, and M.~Mar\'oti~\cite{DJM}. Their result, to be detailed soon, is  equivalent to the description of minimal quasivarieties of semilattices with an automorphism. To recall the result of \cite{DJM} in an economic way, we define two concepts.

In general, the \emph{opposite} of a group $\alg \FF=\langle\FF,\cdot\rangle $ is $\opposit{\alg \FF }=\langle\opposit \FF ,\ast \rangle$, where 
$\opposit \FF =\FF $ and $x\mathop{\ast} y:=y\cdot x$. Note that $\opposit{\alg \FF }=\alg \FF $ in our case since $\alg \FF $ is assumed to be commutative. 
For a $\alg \FF $-semilattice  $\alg A=\langle A;\wedge, \FF \rangle$, the  \emph{opposite} of  $\alg A$ is $\opposit{\alg A}=\langle A;\wedge,\opposit \FF \rangle$, where 
$\opposit g(x)=g^{-1}(x)$ for $g\in \opposit \FF =\FF $ and $x\in A$. Then  $\opposit{\alg A}$ is a $\opposit{\alg \FF }$-semilattice (even without assuming the commutativity of $\FF$), and it can be different from $\alg A$ (even when $\alg \FF $ is commutative).

Let $n\in\mathbb N$, and let $\alg A=\langle A;\wedge,\cikl\infty\rangle$ be a $\alg{\cikl\infty}$-semilattice. Remember that $\alg{\cikl\infty}$ is generated by $g$.
 We define a new $\alg{\cikl\infty}$-semilattice $\twm{\alg A}n$ , the \emph{$n$-fold twisted multiple} of $\alg A$, as follows.
\begin{equation}\label{nfoldtwm}
\begin{aligned}
\twm{\alg A}n=\bigl\langle \set o&\cup 
\set{\pair ai: a\in A,\,\, 0\leq i < n};
\wedge, \cikl\infty
\bigr\rangle,\text{ where} \cr
\pair ai\wedge \pair bj&=\begin{cases}\pair {a\wedge b}i& \text{if }i=j,\cr
o&\text{if }i\neq j
  \end{cases}, \quad \pair ai\wedge o=o\wedge o=o, \cr
g(\pair ai)&=\begin{cases}\pair a{i+1}& \text{if }i<n-1,\cr
\pair{g(a)}0&\text{if }i=n-1
  \end{cases},\quad g(o)=o\text.
\end{aligned}
\end{equation}
The trivial (that is, one-element)  $\algcikl \infty$-semilattice is denoted by $\mathbf o$. We define the following  $\algcikl\infty$-semilattices. 

\begin{allowdisplaybreaks}
\begin{enumeratei}
\item\label{DJma} $\alg A_k=\twm {\mathbf o} k$ for $k\in \mathbb N$. 
\item\label{DJmb} $\alg A_\infty=\langle \mathbb Z\cup\set o;\wedge, \FF \rangle$, where $g(o)=o$, $g(i)=i+1$, $o$ is the zero element of the semilattice reduct, and $i\wedge j=o$ for $i\neq j\in\mathbb N$.
\item\label{DJmc} $\alg B_1^+ =\langle\mathbb Z;\min, \cikl\infty\rangle$, where $g(i)=i+1$.
\item\label{DJmd} $\alg B_k^+= \twm {\mathbf B_1^+} k$ for $2\leq k\in \mathbb N$.
\item\label{DJme} $\alg B_1^- =\langle\mathbb Z;\min, \cikl\infty\rangle$, where $g(i)=i-1$; note that $\alg B_1^-=\opposit{(\alg B_1^+)}$.
\item\label{DJmf} $\alg B_k^-= \twm {\mathbf B_1^-} k$ for $2\leq k\in \mathbb N$; note that $\alg B_k^-=\opposit{(\alg B_k^+)}$.
\item $\alg C_1=\bigl\langle\set{\pair xy \in\mathbf Z^2: x\leq y};\wedge, \cikl\infty \bigr\rangle$, where $g(\pair xy)=\pair{x+1}{y+1}$ and 
$\pair{x_1}{y_1}\wedge \pair{x_2}{y_2} = 
\pair{\min(x_1,x_2)}{\max(y_1,y_2)}$. The Hasse diagram of the semilattice reduct is depicted in Figure~\ref{figone}, and $g$ is the shift operation to the right by one unit (that is, by the unit vector given in the figure).
\item\label{DJmg} $\alg C_k= \twm {\alg C_1}k$ for $2\leq k\in\mathbb N$.
\end{enumeratei}
\end{allowdisplaybreaks}

With reference to the list above, now we are ready to recall the main result of  W.~Dziobiak,~J.~Je\v zek, and M.~Mar\'oti~\cite{DJM}.

\begin{theorem}[\cite{DJM}]\label{djezmThm} The minimal quasivarieties of $\algcikl \infty$-algebras are precisely the
quasivarieties generated by one of the 
algebras \eqref{DJma}--\eqref{DJmg}. These minimal quasivarieties are pairwise distinct.
\end{theorem}

\section{More about $\alg\FF$-semilattices}\label{secmorabout}
Let us agree that  $\alg\FF$ always denotes an abelian group, and $\slat\FF$ 
stands for the variety of $\alg\FF$-semilattices. For $\alg A\in\slat\FF $, if $(A;\wedge)$ has a zero element (in other words, a least element), then it is unique and we  denote it by $o$. As rule, none of the formulas $A\setminus\set o$ and $a\neq o$  implies that $\alg A$ has a zero. (If $\alg A$ has no zero, then $a\neq o$ means no condition on $a$ and  $A\setminus\set o =A$.)
A \emph{$1$-generated}  
(or \emph{cyclic}) $\alg\FF$-semilattice is a $\alg\FF$-semilattice generated by a single element. 
The following two lemmas were stated for $\alg\FF=\algcikl\infty$ in  W.\ Dziobiak,  J.\ Je\v zek, and M.\ Mar\'oti~\cite{DJM}; their proofs are presented for the reader's convenience.

\begin{lemma}\label{l:f1} Assume that $\alg\FF$ is an abelian group and $t$ is a unary $\alg\FF$-semilattice term. Then the following assertions hold.
\begin{enumeratei}
\item\label{l:f1a} The group $\alg\FF$ has a finite nonempty subset $H$ such $\slat\FF $ satisfies the identity 
$ t(x)\approx \bigwedge \set{h(x):h\in H}$.
\item\label{l:f1b} For every $\alg A\in\slat\FF$,   $t_{\alg A}\colon \alg A\to\alg A$ is an endomorphism. Hence, for every unary $\slat\FF $-term $s$, $\slat\FF $ satisfies the identity $s(t(x))\approx t(s(x))$.
\item\label{l:f1c} If $\alg A\in \slat\FF $ has a zero element $o$, then $t(o)=o$ in $\alg A$.
\item\label{l:f1d} If $\alg\FF$ is finite, then every finitely generated $\alg\FF$-semilattice is finite and has a zero element.
\end{enumeratei}
\end{lemma}

\begin{proof} Since any two basic operations commute,  \eqref{l:f1b} is clear; it also follows from \'A.~Szendrei~\cite[Proposition 1.1]{SzA}. 
This implies \eqref{l:f1a} by  induction on the length of $t$. Assume that $o\in\alg A\in\slat\FF $. Since $o$ is a fixed point of every automorphism of  $\langle A;\wedge\rangle$ and $o\wedge o=o$, \eqref{l:f1c} and \eqref{l:f1d} follow from \eqref{l:f1a}. 
\end{proof}

\begin{lemma}\label{l:f3} Assume $\alg A\in\slat\FF$ is generated by an element $a\in A$. Let $s$, $s_1$ and $s_2$ be unary $\slat\FF $-terms. Then 
\begin{enumeratei}
\item\label{l:f3a} if $s_1(a)=s_2(a)$, then the identity $s_1(x)\approx s_2(x)$ holds in $\var V(\alg A)$; 
\item\label{l:f3b} if $s(a)=o$, 
the zero element  of $\alg A$, then the identity $s(x)\wedge y\approx s(x)$ holds in $\var V(\alg A)$. 
\end{enumeratei}
\end{lemma}

\begin{proof} 
Assume that $s_1(a)=s_2(a)$, and let $b\in A$. Since $\alg A$ is generated by $a$,  $b$ is of the form $t(a)$ for some unary term $t$. Using Lemma~\ref{l:f1}\eqref{l:f1b}, $s_1(b)=s_1(t(a))=t(s_1(a))=t(s_2(a))=s_2(t(a))=s_2(b)$. Hence, \eqref{l:f3a} holds.

To prove \eqref{l:f3b}, assume that $s(a)=o$ in $\alg A$, and let $b,c\in A$. Pick a unary term $t$ such that $b=t(a)$. It follows from Lemma~\ref{l:f1}\eqref{l:f1b}-\eqref{l:f1c} that $s(b)=s(t(a))=t(s(a))=t(o)=o$. Hence, $s(b)\wedge c=s(b)$. Consequently, the identity $s(x)\wedge y \approx s(x)$ holds in $\alg A$, and also in $\var V(\alg A)$. 
\end{proof}

The concept of a semilattice over $\alg\FF$ is analogous to that of a vector space over a field. The  following two statements indicate that this analogy is quite strong in the ``one-dimensional case''. 
If $\alg A$ is a $\alg\FF$-semilattice and $b\in A$, then $\genpsub{\alg A}b$ or simply $\gensub b$ denotes the \emph{subalgebra generated} by $b$. For $B=\gensub b$,  we can also write $\alg B=\gensub b$ if we consider $\gensub b$ an algebra rather than a subset. 

\begin{corollary}\label{l:sz} Let $\alg A$ be a $\alg\FF$-semilattice generated by an element $a$. Then  $\alg A$ is a  free algebra in $\var V(\alg A)$,  freely generated by~$a$.
\end{corollary}

\begin{proof} Consider an arbitrary $b\in \alg B\in\var V(\alg A)$. Define a map $\phi\colon \alg A\to \alg B$ by $s(a)\mapsto s(b)$, where $s$ denotes a unary term. Then $\phi$ is a well-defined map by Lemma~\ref{l:f3}\eqref{l:f3a}. It is a homomorphism since 
$\phi(s(a))=s(b)=s(\phi(a))$ and $\phi(s(a)\wedge r(a))=s(b)\wedge r(b)=\phi(s(a))\wedge \phi(r(a))$. Clearly, $\phi$ extends the $\set{a}\to\set{b}$ map. Hence, $\set{a}$ freely generates $\alg A$.
\end{proof}

\begin{lemma}\label{lnewie}
Let $\alg A$ be a \emph{$1$-generated} $\alg \FF$-semilattice, and assume that $\set{d}$ is a subalgebra of $\alg A$. Then $d$ is the zero element of $\alg A$.
\end{lemma}

\begin{proof} Let  $\alg A=\genpsub{ \alg A}a$. By Lemma~\ref{l:f1}\eqref{l:f1a}, $d\leq h(a)$ for some $h\in \alg \FF$. Using that $\set d$ is a subalgebra, we obtain that $d=h^{-1}(d)\leq h^{-1}(h(a))=a$. Hence $a$ belongs to the order filter $\filter d$ generated by $d$. Since this filter is clearly a subalgebra and contains $a$, we conclude that $A=\filter d$. Thus $d=o$.
\end{proof}

Since there is 
only  one way, the ``$\alg\FF$-trivial way'', to expand the two-element meet-semilattice $\langle\set{0,1},\leq\rangle$ into a $\alg\FF$-semilattice, we can speak of \emph{the} two-element $\alg\FF$-semilattice. Although a $1$-generated $\alg\FF$-semilattice $\alg B\in \var Q(\alg A)$ is  free in $\var V(\alg B)$ and in $\var Q(\alg B)$ by Corollary~\ref{l:sz}, it is not necessarily free in $\var Q(\alg A)$.

\begin{lemma}\label{l:k} Let $\alg A$ be a $1$-generated $\alg\FF$-semilattice that is isomorphic to each of its nontrivial $1$-generated subalgebras.  
Then $\var Q(\alg A)$ does not contain the two-element $\alg\FF$-semilattice. Moreover, if $\alg A=\genpsub{\alg A}a$ and $s$ and $t$ are unary terms with $s(a)\neq t(a)$, then the quasivariety $\var Q(\alg A)$ satisfies the quasi-identity $s(x)\approx t(x)\then x\approx x\wedge y$.
\end{lemma}

\begin{proof}  We can assume that $\alg A=\gensub a$ is nontrivial. Let $s$ and $t$ be unary terms such that $s(a)\neq t(a)$. There are such terms since $|A|\neq 1$. We claim that 
\begin{equation}\label{lseiW}
\text{$s(b)\neq t(b)$ for all $b\in \alg A\setminus\set 0$.}
\end{equation} 
To obtain a contradiction, suppose that $b\in  \alg A\setminus\set 0$ such that $s(b)=t(b)$. 
Then, by Lemma~\ref{l:f3}\eqref{l:f3a}, the subalgebra $\gensub b$ satisfies the identity $s(x)\approx t(x)$. Moreover, $\gensub b$  is a nontrivial subalgebra by Lemma~\ref{lnewie}. Thus it is  isomorphic to $\alg A$ by the assumption.  Therefore, the identity $s(x)\approx t(x)$ also holds in $\alg A$, which contradicts $s(a)\neq t(a)$. This proves \eqref{lseiW}. 
Next, it follows from \eqref{lseiW} that $\alg A$ satisfies the quasi-identity $s(x)\approx t(x)\then x\approx x\wedge y$. So does $\var Q(\alg A)$.  The evaluation $(x,y)=(1,0)$ shows that this quasi-identity fails in the two-element $\alg\FF$-semilattice, whence this $\alg\FF$-semilattice does not belong to $\var Q(\alg A)$.
\end{proof}

The largest congruence and the least congruence of an algebra $\alg A$ are denoted by $\nabla=\nabla_{\alg A}$ and $\Delta=\Delta_{\alg A}$, respectively.

\begin{lemma}\label{l:kong} Let $\alg A$ be a nontrivial $1$-generated $\alg\FF$-semilattice such that every $b\in  A\setminus \set o$ generates a subalgebra isomorphic to $\alg A$. 
Assume that $\rho$ is a congruence of $\alg A$ and $\rho\notin\set{\nabla_{\alg A}, \Delta_{\alg A}}$. Then the quotient algebra $\alg A/\rho$ does not belong to the quasivariety $\var Q(\alg A)$.
\end{lemma}

\begin{proof}
Since $\rho\neq\Delta$, we can pick a pair $(b,c)\in \rho$ such that $b\neq c$. Pick an element $a\in A$ that generates $\alg A$. Then there exist unary terms $s$ and $t$ such that $b=s(a)$ and $c=t(a)$. 
Since $\alg A/\rho$ is generated by $a/\rho$ and $s(a/\rho)=s(a)/\rho=t(a)/\rho=t(a/\rho)$, we obtain from Lemma~\ref{l:f3}\eqref{l:f3a} that the identity $s(x)\approx t(x)$ holds in $\alg A/\rho$. But $\alg A/\rho$ is nontrivial since $\rho\neq\nabla$, whence the quasi-identity $s(x)\approx t(x)\then x\approx x\wedge y$ fails in 
$\alg A/\rho$. On the other hand, this quasi-identity holds in $\alg A$ and also in $\var Q(\alg A)$   by Lemma~\ref{l:k}. 
Thus $\alg A/\rho\notin \var Q(\alg A)$.
\end{proof}

\section{The $1$-generated free algebras of minimal quasivarieties}\label{secMain}
The minimal  quasivarieties of $\alg \FF$-semilattices are described by the following theorem; except for the obvious $\alg\FF$-trivial case, it suffices to deal with $1$-generated nontrivial $\alg \FF$-semilattices.

\begin{theorem}\label{t:eq} \textup{(A)} The variety $\var S$ of all $\alg\FF$-trivial $\alg\FF$-semilattices is a minimal quasivariety, and it is generated by the two-element $\alg\FF$-semilattice. Except for $\var S$, each minimal quasivariety of $\alg\FF$-semilat\-tices is generated by a $1$-gener\-ated $\alg\FF$-semilattice. Furthermore, $\var S$ is the only minimal quasivariety  of $\alg\FF$-semi\-lattices whose $1$-generated free algebra is one-element.
 
\textup{(B)} Let $\alg A$ be a nontrivial $1$-generated $\alg \FF$-semilattice, and let $a\in A$ be a fixed element that generates $\alg A$. Then the following four conditions are equivalent.
\begin{enumeratei}
\item\label{t:eqa} $\var Q(\alg A)$ is a minimal quasivariety.
\item\label{t:eqb} For each $b\in A$, if the subalgebra $B$ generated by $b$ is not a singleton, then there is an isomorphism $\phi\colon\alg A\to \alg B$ such that $\phi(a)=b$.
\item\label{t:eqc} $\alg A$ is isomorphic to each if its nontrivial $1$-generated subalgebras.
\item\label{t:eqd} Each nonzero element of $\alg A$ generates a subalgebra isomorphic to $\alg A$.
\end{enumeratei}
Moreover, if some $($equivalently, each$)$ of the conditions \eqref{t:eqa}, \eqref{t:eqb}, \eqref{t:eqc}, and \eqref{t:eqd} holds, then 
\begin{enumeratei}
\setcounter{enumi}{4}
\item\label{t:eqe} all nontrivial $1$-generated algebras of $\var Q(\alg A)$ are $($isomorphic to$)$ the free algebra $\freealg{\var Q(\alg A)}1$ of $\var Q(\alg A)$.
\end{enumeratei}
\end{theorem}

The set of minimal quasivarieties of $\alg\FF$-semilattices will often be denoted by $\minqvar{\alg\FF}$. (Since the quasi-identities in the language of $\alg\FF$-semilattices form a set, so do the minimal quasivarieties.)

\begin{proof}[Proof of Theorem~\ref{t:eq}] Since $\alg\FF$-trivial $\alg\FF$-semilattices are essentially (that is, up to term equivalence) semilattices, it belongs to the folklore that $\var S$ is generated by the two-element $\alg\FF$-semilattice. Since $\freealg{\var S}1$ is one-element, it does not generate $\var S$. Assume that $\var U\in \minqvar{\alg\FF}$ such that $\var U$ is not $\alg\FF$-trivial. Then there are an algebra $\alg B\in \var U$, an element $b\in B$, and a group element $g\in \FF$ such that $g(b)\neq b$. Hence the subalgebra generated by $b$ has at least two elements, which implies that $\freealg{\var U}1$ not a singleton. Therefore,  by the minimality of $\var U$, 
$\freealg{\var U}1$ generates $\var U$. This proves part (A).

Next, we deal with part (B). Clearly, \eqref{t:eqb}  implies \eqref{t:eqc}. It follows from  
Lemma~\ref{lnewie} that \eqref{t:eqc} and \eqref{t:eqd} are equivalent.

To prove that \eqref{t:eqa} implies \eqref{t:eqb}, assume that \eqref{t:eqa} holds. 
Let $b\in A$, and assume that the subalgebra $\gensub b$ is not a singleton. Define a map $\phi\colon \alg A\to \alg B$ by the rule $r(a)\mapsto r(b)$, where $r$ ranges over the set of unary terms. Letting $r$ be the identity map, we obtain that $\phi(a)=b$. 
 We conclude from Corollary~\ref{l:sz} that $\phi$ is a homomorphism, and it is clearly surjective. In order to prove that $\phi$ is injective,  assume that 
$r$ and $s$ are unary terms such that 
$\phi(r(a))=\phi(s(a))$, that is, $r(b)=s(b)$. By Lemma~\ref{l:f3}\eqref{l:f3a}, the identity $r(x)\approx s(x)$ holds in $\var V(\alg B)$, whence it also holds in $\var Q(\alg B)$.  Since $\alg B\in\var Q(\alg A)$ and $\alg B$ is nontrivial, the minimality of  $\var Q(\alg A)$ implies $\var Q(\alg B)=\var Q(\alg A)$. Thus the identity $r(x)\approx s(x)$ holds in $\alg A$, and we obtain that $r(a)=s(a)$. Hence, $\phi$ is injective, and it is an isomorphism. Therefore, \eqref{t:eqa} implies \eqref{t:eqb}.

Next, to show that \eqref{t:eqc} implies \eqref{t:eqe}, assume that \eqref{t:eqc} holds. Let $\alg B\in\var Q(\alg A)$ be a nontrivial $1$-generated algebra generated by $b\in B$.  By Corollary~\ref{l:sz}, there exists a (unique) surjective homomorphism $\phi\colon \alg A\to \alg B$ such that $\phi(a)=b$. Suppose, to derive a contradiction, that $\phi$ is not an isomorphism. Then $\phi$ is not injective, whence $\ker\phi\neq\Delta$. Since $\alg B$ is nontrivial, $\ker\phi\neq\nabla$. Thus Lemma~\ref{l:kong} (together with Lemma~\ref{lnewie}) applies, and we obtain that $\alg B\cong \alg A/\ker\phi\notin \var Q(\alg A)$. This contradicts the assumption on $\alg B$, and we conclude that \eqref{t:eqc} implies \eqref{t:eqe}.

Finally, to show that  \eqref{t:eqc} implies \eqref{t:eqa}, assume that \eqref{t:eqc} holds.  Let $\var K$ be a nontrivial subquasivariety of  $\var Q(\alg A)$. 
Let $\alg B$ denote the  free algebra $\freealg{\var K}b$. For the sake of contradiction, suppose that $\alg B$ is one-element. Then $g(b)=b$ for every $g\in\alg\FF$, and we conclude that the identity $g(x)\approx x$ holds in $\var K$. 
Therefore, the two-element $\alg\FF$-semilattice belongs to $\var K\subseteq \var Q(\alg A)$, which contradicts Lemma~\ref{l:k}. Thus $\alg B\in \var K$ is a nontrivial $1$-generated algebra in $\var Q(\alg A)$. 
Since we already know that \eqref{t:eqc} implies \eqref{t:eqe}, we obtain from \eqref{t:eqe}  that $\alg A\cong \alg B\in \var K$. Thus $\alg A\in\var K$, yielding that   $\var Q(\alg A)\subseteq \var K$. This means that $\var Q(\alg A)\in \minqvar{\alg\FF}$. Consequently, \eqref{t:eqc} implies \eqref{t:eqa}.
\end{proof}

\section{Two constructs}\label{seCtwoconstr}
The following construct is due to M.~Mar\'oti~\cite{MM}. Let $\sgrp$ be a subgroup (that is, a nonempty subuniverse) of $\alg\FF$. The \emph{Mar\'oti semilattice over $\alg\FF$} is
\begin{align*}
\mmalg= \langle \mmset;\cap, \FF\rangle\text{, where }
\mmset =\emptyset\cup \set{g\sgrp: g\in \FF},
\end{align*}
 $\cap$ is the usual intersection, and $f(g\sgrp)$ for $f,g\in\FF$ is defined as $fg\sgrp$.  Note that $\mmalg$ consists of atoms and a zero; the atoms are the (left) cosets of $\sgrp$ while the emptyset is $o$. 
Clearly, $\mmalg$ is a $\alg\FF$-semilattice. 
Note that $\alg A_\infty$ in 
Theorem~\ref{djezmThm} is (isomorphic to) ${\qalgeb{\set1}{\algcikl \infty}}$.
The importance of this $\alg\FF$-semilattice is explained by the following statement.

\begin{proposition}\label{soZRc}
\textup{(A)}
For every subgroup $\sgrp$ of an abelian group $\alg\FF$, $\mmalg$ 
generates a minimal quasivariety $\var K$. If $\sgrp=\FF$, then  $\mmalg$ is the two-element $\alg\FF$-semi\-lattice. Otherwise, $\mmalg\cong \freealg{\var K}1$.

\textup{(B)} Let $\sgrp_1$ and $\sgrp_2$ be subgroups of $\alg\FF$. Then the algebras 
$\qalgeb{\sgrp_1}{\alg\FF}$ and $\qalgeb{\sgrp_2}{\alg\FF}$ generate the same minimal quasivarieties of $\alg\FF$-semilattices if and only if $\sgrp_1=\sgrp_2$.  
\end{proposition}

\begin{proof} Since the case of $\sgrp=\FF$ is trivial, we assume that $\sgrp$ is a proper subgroup. Then there is an $f\in \FF\setminus \sgrp$. Since $o=\emptyset = 1\sgrp\cap f\sgrp$ and, for any $g,h\in\FF$, $h\sgrp=(hg^{-1})(g\sgrp)$, $\mmalg$ is generated by each of its nonzero elements. Thus part (A) follows from Theorem~\ref{t:eq}(B).

To prove part (B), we can assume that both subgroups in question are proper. 
The ``if'' part is obvious. To prove the converse implication, assume that $\qalgeb{\sgrp_1}{\alg\FF}$ and $\qalgeb{\sgrp_2}{\alg\FF}$ generate the same quasivariety $\var K$.  Then $\qalgeb{\sgrp_1}{\alg\FF}\cong \qalgeb{\sgrp_2}{\alg\FF}$ since both are isomorphic to $\freealg{\var K}1$ by part (A). Hence, there is an isomorphism $\phi\colon \qalgeb{\sgrp_1}{\alg\FF}\to \qalgeb{\sgrp_2}{\alg\FF}$. Since $\sgrp_1=1\sgrp_1$ is an atom, so is its $\phi$-image. Thus there is an $f\in\FF$ such that $\phi(\sgrp_1)=f\sgrp_2$. Let $x$ denote an arbitrary element of $\FF$. Using that  $\phi$ is an isomorphism,  we obtain that
\begin{align*}
x\in H_1 &\iff xH_1=H_1 \iff \phi(xH_1)=\phi(H_1) \iff x\phi(H_1)=fH_2\cr
&\iff  xfH_2=fH_2\iff xff^{-1}\in H_2\iff x\in H_2,
\end{align*}
which means that $H_1=H_2$.
\end{proof}
Let $\Sub{\alg \FF}$ denote the set of all subgroups (that is, nonempty subuniverses) of $\alg \FF$. The following statement clearly follows from 
Proposition~\ref{soZRc}.

\begin{corollary} \textup{(A)} For every abelian group $\alg\FF$, 
there are at least $|\Sub{\alg \FF}|$ many minimal quasivarieties of $\alg\FF$-semilattices. 

\textup{(B)} For each cardinal $\kappa$, there exists an abelian group $\alg\FF$ such that there are at least $\kappa$ minimal quasivarieties of $\alg\FF$-semilattices.
\end{corollary}

The next construction generalizes the $n$-fold twisted multiple construct, see \eqref{nfoldtwm}. 
Let $K$ be a subgroup of $\alg\FF$, and choose a system $\rep$ of representatives of the (left) cosets of $K$. That is,  $\rep\subseteq \FF$ such that $|\rep\cap gK|=1$ holds for all $g\in \FF$. Usually (but not in Lemma~\ref{MicoDieAE}), we  assume that $T\cap K=\set 1$.
(We will prove that, up to isomorphism, our construct does not depend on the choice of $\rep$.)
Assume that $\alg U$ is a $\alg K$-semilattice. We define a $\alg\FF$-semilattice $\gtwm K{\FF}U$ (up to isomorphism) as follows. 
\begin{equation}\label{nfCPlTwm}
\begin{aligned}
\gtwm K{\FF}U =\bigl\langle \set o&\cup 
(U\times \rep) ;
\wedge, \FF
\bigr\rangle,\qquad\text{where, for }g\in\FF, 
\cr   
g(\pair ut)&= \pair{gtf^{-1}(u)}{f} \text{ if }f \in \rep\cap gtK,
\cr
 g(o)&=o,\quad (u,t)\wedge o=o\wedge o=o, \text{ and } \cr
\pair {u_1}{t_1}\wedge \pair {u_2}{t_2}&=\begin{cases}\pair{u_1\wedge u_2}{t_1}& \text{if }t_1=t_2,\cr
o&\text{if }t_1\neq t_2
  \end{cases} \text{.}
\end{aligned}
\end{equation}

\begin{lemma}\label{MicoDieAE} 
The isomorphism class  
of the algebra $\gtwm K{\FF}U$ does not depend on the choice of $\rep$. Furthermore, 
$\gtwm K{\FF}U$ is a $\alg\FF$-semilattice.  
\end{lemma}

\begin{proof} 
First, we prove that $\gtwm K{\FF}U$ is a $\alg\FF$-semilattice. 
Assume that $\pair ut\in U\times\rep$ and $g_1,g_2\in \FF$. Denote by $f_1$ and $f_2$ the unique element of $\rep\cap g_2tK$ and  $\rep\cap g_1f_1K$, respectively. Then 
$f_1K=g_2tK$ yields that $f_2\in g_1f_1K=g_1g_2tK$. Using the commutativity of $\alg\FF$ at $\csillegy$, we obtain that 
\begin{equation*}
\begin{aligned}
g_1\bigl(g_2(\pair ut)\bigr)= g_1(\pair{g_2tf_1^{-1}(u)}{f_1}) = 
\pair {g_1f_1f_2^{-1}(g_2tf_1^{-1}(u))} {f_2} \cr
= \pair{g_1f_1f_2^{-1}g_2tf_1^{-1}(u)}{f_2} \csillegy \pair{g_1g_2tf_2^{-1}(u)}{f_2}=(g_1g_2)(\pair ut)\text.
\end{aligned}
\end{equation*}
This implies that axiom \eqref{Fsldefc} holds. To prove the validity of \eqref{Fsldefd}, let $g\in \FF$, $u_1,u_2,\in U$ and $t,t_1,t_2\in \rep$ such that $t_1\neq t_2$. Let $f\in \rep\cap gtK$, $f_1\in \rep\cap gt_1K$, and  $f_2\in \rep\cap gt_2K$. Since $t_1$ and $t_2$ belong to distinct cosets of $K$, we obtain that $gt_1K\neq gt_2K$ and  $f_1\neq f_2$. Thus
\begin{align*} g(\pair{u_1}{t_1})\wedge g(\pair{u_2}{t_2}) &= \pair{gt_1f_1^{-1}(u_1)}{f_1}\wedge \pair{gt_2f_2^{-1}(u_2)}{f_2}\cr  &=o=g(o)=g\bigl(\pair{u_1}{t_1}\wedge \pair{u_2}{t_2}  \bigr)\text. 
\end{align*}
Since we also obtain that 
\begin{align*} g(\pair{u_1}{t})\wedge g(\pair{u_2}{t}) &= \pair{gtf^{-1}(u_1)}{f}\wedge \pair{gtf^{-1}(u_2)}{f}\cr
  & =  \pair{ gtf^{-1}(u_1)\wedge gtf^{-1}(u_2) }{f}  =\pair{ gtf^{-1}(u_1\wedge u_2)}{f}\cr  
&= g(\pair{u_1\wedge u_2}{t}) 
= g(\pair{u_1}{t} \wedge \pair{u_2}{t}  ), 
\end{align*}
axiom \eqref{Fsldefd} also holds. The rest of the axioms trivially hold, whence the algebra $\gtwm K{\FF}U$ is a $\alg\FF$-semilattice.  

Second, let $\alg B$ be the algebra $\gtwm K{\FF}U$ defined above with $T$. Let $T'$ be another system of representatives of the cosets of $\alg K$, and denote by $\alg B'$ the algebra $\gtwm K{\FF}U$ constructed with $T'$ instead of $T$. For $t\in T$, the unique element of $T'\cap tK$ is denoted by $t'$. 
We claim that the map
$\phi\colon \alg B\to \alg B'$, defined by
\begin{equation*}
o\mapsto o,\qquad \pair ut\mapsto \pair{t'^{-1}t(u)}{t'}
\end{equation*}
is an isomorphism. 
It is clearly a semilattice isomorphism since $t'^{-1}t$ induces an automorphism of $\langle U;\wedge\rangle$. Assume that $g\in \FF$, $\pair{u}{t_1}\in B$,  and $\pair{v}{t_1'}\in  B'$. Let $T\cap gt_1K=\set{t_2}$. Then $t_2 \in (gK)(t_1K)=(gK)(t_1'K)=gt_1'K $, which implies that $T'\cap gt_1'K=\set{t_2'}$. Hence, we obtain that 
\begin{align}
g_{\alg B}(\pair u{t_1})&= \pair{gt_1t_2^{-1}(u)}{t_2}\text{ in }\alg B\text{ and}\cr
g_{\alg B'}(\pair v{t_1'})&= \pair{gt_1't_2'^{-1}(v)}{t_2'}\text{ in }\alg B'\text.\label{lsZeW}
\end{align}
Observe that $t_1'^{-1}t_1\in K$. 
Using  \eqref{lsZeW} with $v=t_1'^{-1}t_1(u)$, we obtain that $\phi$ preserves the operation $g$ since%
\begin{align*}
g_{\alg B'}( \phi\pair u{t_1}) &= g_{\alg B'}(\pair {t_1'^{-1}t_1(u)}{t_1'})
= \pair{gt_1't_2'^{-1}(t_1'^{-1}t_1(u))}{t_2'}
\cr
&= \pair{gt_1't_2'^{-1}t_1'^{-1}t_1(u)}{t_2'}
=
 \pair{t_2'^{-1} t_2 g t_1  t_2^{-1}  (u)}{t_2'}  \cr 
& = 
\pair{t_2'^{-1} t_2 ( g t_1  t_2^{-1}  (u))}{t_2'} =  \phi\bigl( \pair{ g t_1  t_2^{-1}  (u)}{t_2}  \bigr)
\cr
&=\phi(g_{\alg B}( \pair u{t_1}  ) )\text. \qedhere
\end{align*}
\end{proof}

\begin{remark}  The $n$-fold twisted multiple construct is indeed a particular case of \eqref{nfCPlTwm}, up to term equivalence. To see this, assume that $\alg\FF_1$ and $\alg\FF_2$ are abelian groups and  $\phi\colon \alg\FF_2\to \alg\FF_1$ is a surjective homomorphism. Then each $\alg\FF_1$-semilattice $\alg B_1=\langle B;\wedge,\FF_1\rangle$ becomes a $\alg\FF_2$-semilattice $\alg B_2=\langle B;\wedge,\FF_2\rangle$ \emph{by change of groups}:  $\alg B_2=\langle B;\wedge,\FF_2\rangle$ and, for $g\in\FF_2$ and $b\in B$, we define $g(b)=(\phi(g))(b)$. Clearly, $\alg B_1$ and $\alg B_2$ are term equivalent. 
Now, let $n\in\mathbb N$, and let $\alg A=\langle A;\wedge,\cikl\infty\rangle$ be a $\alg{\cikl\infty}$-semilattice, where $\alg{\cikl\infty}$ is generated by $g$. Take the subgroup $K=\set{g^{nk}:k\in\mathbb Z}$, and let $\rep=\set{g^0,g^1,\ldots,g^{n-1}}$. Using the group isomorphism $\phi\colon \alg K\to \alg{\cikl\infty}$, defined by $\phi(g^{nk})=g^k$, $\alg A$ becomes a $\alg K$-semilattice $\alg A'$ by change of groups. We claim that $\twm{\alg A}n$ is isomorphic to $\gtwm K {{\alg{\cikl\infty}}}{A'}$; the easy proof is omitted since we will not use this fact.
\end{remark}

\begin{remark}\label{remfanshaped}
The Mar\'oti semilattice over $\alg\FF$ is also a particular case of \eqref{nfCPlTwm} since $\qalgeb K{\alg\FF}  \cong \gtwm K{\FF}o$, where $\alg{ o}=\langle\set o;\wedge,\FF\rangle$.
\end{remark}

\begin{lemma}\label{MdHlW} 
Besides the assumptions of Lemma~\ref{MicoDieAE}, assume that $K$ is a proper subgroup of $\alg\FF$, and that
$\alg U=\gensub a$ is a $1$-generated $\alg K$-semilattice. Then, for every $t\in T$, 
$\pair at$ generates the $\alg\FF$-semilattice $\gtwm K{\FF}U$.  
\end{lemma}

\begin{proof}
Let $B=\genpsub{\gtwm K{\FF}U}{\pair at}$. It suffices to show that $U\times T\subseteq B$ since then $o\in B$ follows from $|T|>1$.
If $g\in K$, then $t\in tK\cap T=tgK\cap T=gtK\cap T$, and we obtain that $g(\pair at)=\pair{gtt^{-1}(a)}{t}=\pair{g(a)}t$. 
  Hence, it follows from Lemma~\ref{l:f1}\eqref{l:f1a} that $\pair ut\in B$ for all $u\in U$. Now, let $f\in T$. Then 
$f\in fK\cap T=t^{-1}ftK\cap T$, whence 
$(t^{-1}f)(\pair ut)=
\pair{(t^{-1}f)tf^{-1}(u)}f=\pair uf$.
Thus $\pair uf\in B$. Hence, $U\times T\subseteq B$, as desired.  
\end{proof}

\section{Minimal quasivarieties of semilattices over finite abelian groups}
\label{seCfinigrp}
Given a minimal quasivariety $\var R$  of $\alg\FF$-semilattices, we associate a subgroup $\sgrp_{\var R}$ with $\var R$ as follows. Take the free algebra $\freealg{\var R}a$, and define $\sgrp_{\var R}$ as the \emph{stabilizer} of $a$, that is,  $\sgrp_{\var R}=\set{g\in\FF: g(a)=a}$. The minimal quasivarieties of  semilattices over \emph{finite} abelian groups are satisfactorily described by the following theorem.

\begin{theorem}\label{fIndSrpt} Let $\alg\FF$ be a \emph{finite} abelian group. Then the map $\alpha\colon \var R\mapsto \sgrp_{\var R}$ is a bijection from the set $\minqvar{\alg\FF}$ of minimal quasivarieties to $\Sub{\alg \FF}$. The inverse map $\beta\colon \Sub{\alg \FF}\to \minqvar{\alg\FF}$   is   defined by  $H\mapsto \var Q\bigl( \mmalg \bigr)$. 
\end{theorem}

\begin{proof} Since the (minimal) quasivariety of $\alg\FF$-trivial $\alg\FF$-semilattices clearly corresponds to the case $\sgrp =\FF$ and vice versa, we can assume that $\var R$ is not $\alg\FF$-trivial and that $\sgrp$ is a proper subgroup. 
We know from Proposition~\ref{soZRc}  that $\beta$ is an injective map 
from $\Sub{\alg\FF}$ to $\minqvar{\alg \FF}$. Obviously, $\alpha$ is a map from 
$\minqvar{\alg \FF}$ to $\Sub{\alg\FF}$. 

To prove that the composite map $\alpha\circ\beta$ act identically on $\Sub{\alg\FF}$, assume that $\sgrp$ is a proper subgroup of $\alg\FF$. Then $\beta(\sgrp)$ is the quasivariety $\var Q\bigl( \mmalg \bigr)$. By Proposition~\ref{soZRc}(A),  $\mmalg$ is (isomorphic to) the free algebra of rank 1 in this quasivariety. Its free generator is of the form $fH$ for some $f\in \FF$. Hence $\alpha(\beta(\sgrp))$ is the stabilizer of $fH$, that is 
\[\alpha(\beta(\sgrp))=\set{g\in \FF: 
gfH=fH}= \set{g\in \FF: 
gff^{-1}\in H}=H\text.
\]
This proves that  $\alpha\circ\beta$ acts identically on $\Sub{\alg\FF}$ (even when $\alg\FF$ is infinite).

Next, we  prove that $\beta\circ\alpha$ acts identically on $\minqvar{\alg\FF}$. To do so, let $\var R\in\minqvar{\alg\FF}$, distinct from the quasivariety of $\alg\FF$-trivial $\alg\FF$-semilattices. 
Then  $\freealg{\var R}a$ generates $\var R$ by Theorem~\ref{t:eq}(A). 
We know from Lemma~\ref{l:f1}\eqref{l:f1d} that $o\in \freealg{\var R}a$, $\freealg{\var R}a$ is finite, and  $\freealg{\var R}a$ has an atom, $b$. 
Since every $g\in \FF$ preserves the semilattice order, 
the subalgebra $\genpsub{\freealg{\var R}a}b$ consists of some atoms and, possibly, of $o$.
This subalgebra is isomorphic to $\freealg{\var R}a$ by Theorem~\ref{t:eq}(B). Hence, by finiteness, this subalgebra equals $\freeaset{\var R}a$.
It follows that $\freeaset{\var R}a=\set o\cup \set{g(a): a\in \FF}$; note that $g_1(a)=g_2(a)$ may occur with distinct $g_1,g_2\in \FF$. Let $\sgrp=\alpha(\var R)=\set{g\in\FF: g(a)=a}$, the stabilizer of $a$. 
It suffices to show~that  
\begin{equation}\label{mYmV}
\text{$\freealg{\var R}a\cong\mmalg$ }
\end{equation}
since then the minimality of $\var R$ implies that 
\[\beta(\alpha(\var R ))= \beta(\sgrp)=\var Q\bigl(\mmalg\bigr)=\var R\,\text.
\] 
Define a map $\phi\colon\freealg{\var R}a\to\mmalg$ by 
$g(a)\mapsto g\sgrp$ and $o\mapsto \emptyset$. Since $f(a)=g(a)$ if{f} 
$a=f^{-1}g(a)$ if{f} $f^{-1}g\in \sgrp$ if{f} $f\sgrp=g\sgrp$, we obtain that $\phi$ is indeed a map, and it is a bijection. 
Since the nonzero elements are atoms both in $\freealg{\var R}a$ and $\mmalg$, $\phi$ preserves the meet.
Finally, $\phi$ is an isomorphism since, for any $f,g\in\FF$,   
$\phi\bigl(f(g(a))\bigr) = \phi\bigl((fg)(a)\bigr)= (fg)H=f(gH)=f\phi\bigl(g(a)\bigr)\text. 
$
\end{proof}

\section{When $\alg\FF$ is not necessarily finite}\label{secNonFini}
The set $\minqvar{\alg\FF}$ of minimal quasivarieties of $\alg\FF$-semilattices 
splits into three disjoint subsets, 
\begin{align*}
\minqvan{\alg\FF}&=\set{\var R\in \minqvar{\alg\FF}: o\in \freealg{\var R}1\text{, } |\freealg{\var R}1|\geq 2},\cr
\minqvae{\alg\FF}&=\set{\text{the class of $\alg\FF$-trivial $\alg\FF$-semilattices} }\cr
&=\set{\var R\in \minqvar{\alg\FF}: |\freealg{\var R}1|=1}
\text{, and}\cr
\minqvak{\alg\FF}&=\set{\var R\in \minqvar{\alg\FF}: o\notin \freealg{\var R}1}\text.
\end{align*}
For example, by Theorem~\ref{djezmThm},
\begin{align*}
\minqvan{\algcikl \infty}=\set{&
\var Q( \alg A_k ):k\in\mathbb N} \cup \set{\var Q( \alg A_\infty )}\cr
\cup  \set{& \var Q( \alg B_k^{+} ) , \var Q( \alg B_k^{-} ) , \var Q( \alg C_k ):2\leq k\in\mathbb N}\text{ and}
\cr
\minqvak{\algcikl \infty}=
\set{
&\var Q(\alg B_1^+  ) , \var Q( \alg B_1^- ) , \var Q( \alg C_1 ) }\text.
\end{align*}

Let $\var R\in \minqvar{\alg\FF}$ such that $\var R$ is distinct from the quasivariety of $\alg\FF$-trivial $\alg\FF$-semilattices.  Then, by  Lemma~\ref{l:f3}, there are exactly two cases: either $\var R\in \minqvan{\alg\FF}$ and all members of $\var R$ have $o$, or $\var R\in \minqvak{\alg\FF}$ and there are algebras in $\var R$ without $o$. The obvious singleton set $\minqvae{\alg\FF}$ deserves no separate attention.

%
The target of this section is to describe the members of $\minqvan{\alg\FF}$. As we know from Theorem~\ref{t:eq}, they are determined by their free algebras on one generator, that is, by the nontrivial $1$-generated  $\alg\FF$-semilattices $\alg A$ with zero that satisfy condition \eqref{t:eqc} (or \eqref{t:eqd}) of Theorem~\ref{t:eq}.

To give the main definition of this section, assume that  $\var R\in \minqvan{\alg\FF}$ and  $\alg A=\freealg{\var R}a$. Let $K=K_{\var R}$ be the set
$\set{g\in\FF: a\wedge g(a)\neq o}$. (We will show that $K$ is a subgroup of $\alg\FF$.) Let $U=\{t(a): t$ is a unary term in the language of $\alg K$-semilattices$\}$. Then we define $\alg U=\alg U_{\var R}=\langle U;\wedge,K\rangle\text.$

\begin{lemma}\label{l:kdEt} Assume that  $\var R\in \minqvan{\alg\FF}$. Then, for $K=K_{\var R}$ and $\alg U=\alg U_{\var R}$ defined above, the following assertions hold.
\begin{enumeratei}
\item\label{l:kdEta} $K$ is a subgroup of $\alg\FF$.
\item\label{l:kdEtb} Let $f_1,\ldots,f_n\in\FF$. Then $f_1(a)\wedge\cdots\wedge f_n(a)\neq o$ if{f} $f_1K=\cdots =f_nK$.
\item\label{l:kdEtc} $\alg U$ is a $\alg K$-semilattice. If $\alg U$ is nontrivial,  then $\alg U$  has no zero element.
\item\label{l:kdEtd} If $\alg U$ is nontrivial, then $\var Q(\alg U) \in\minqvak{\alg K}$ and $\alg U\cong \freealg{\var Q(\alg U)}1$. 
\item\label{l:kdEte} $\alg A=\freealg{\var R}a$ is isomorphic to  $\gtwm K{\FF}U$.
\end{enumeratei}
\end{lemma}

If $\alg\FF$ is finite, then $\alg U$ is trivial by \eqref {mYmV}. The  case $|U|=1$ is  important even if finiteness is not assumed  since Lemma~\ref{l:kdEt}, together with Remark~\ref{remfanshaped}, clearly implies the following corollary.

\begin{corollary} Let $\var R\in \minqvan{\alg\FF}$. Then  $\alg U_{\var R}$  is 1-element if{f} $\,\freealg{\var R}1$ is $($up to isomorphism$)$   $\qalgeb K{\alg\FF}$ with a proper subgroup $K$ of $\alg \FF$.
\end{corollary}

\begin{proof}[Proof of Lemma~\ref{l:kdEt}]
Since $\alg A$ is nontrivial, $a\neq o$ and $\id\in K$. If $g$ belongs to $K$, then so does $g^{-1}$ since $a\wedge g^{-1}(a)=g^{-1}\bigl(a \wedge g(a)\bigr)\neq o$. To get a contradiction, suppose that $f,g\in K$ but $fg\notin K$. Then $a\wedge f(a)\neq o$ and 
\begin{equation*}
o=a\wedge fg(a)\geq a\wedge f(a)\wedge g(a)\wedge fg(a)=(a\wedge f(a))\wedge g(a\wedge f(a))\text.
\end{equation*} 
It follows from Theorem~\ref{t:eq} that  $a\wedge f(a)$ generates $\alg A=\freealg{\var R}a$. Hence, by applying Lemma~\ref{l:f3}\eqref{l:f3b} to the element $a\wedge f(a)$ and to the term $s(x)=x\wedge g(x)$, we obtain that the identity $x\wedge g(x)\wedge y\approx x\wedge g(x)$ holds in $\var R=\var Q(\alg A)$. Substituting $\pair ao$ for $\pair xy$, we obtain that $o= a\wedge g(a)\wedge o\approx a\wedge g(x)$, which contradicts $g\in K$.
Thus $fg\in K$, and $K$ is a subgroup of $\alg\FF$. This proves \eqref{l:kdEta}.

To prove \eqref{l:kdEtb}, assume that $f_1(a)\wedge\dots\wedge f_n(a)\neq o$, and let $i,j\in\set{1,\ldots,n}$ with $i\neq j$. Then $f_i(a)\wedge f_j(a)\neq o$ since $f_i(a)\wedge f_j(a)\geq f_1(a)\wedge\dots\wedge f_n(a)$. Thus 
$o=f_i^{-1}(o)\neq f_i^{-1}\bigl(f_i(a) \wedge f_j(a) \bigr) =
    a\wedge f_i^{-1}f_j(a)$. Thus $f_i^{-1}f_j\in K$, that is, $f_iK=f_jK$, for all $i,j\in\set{1,\ldots,n}$. This proves the ``only if'' part of \eqref{l:kdEtb}.

To prove the ``if'' part,  we claim 
that, for all $n\in \mathbb N$,
\begin{equation}\label{vkvnX}
\text{if }g_1,\ldots,g_n\in K\text{, then } a\wedge g_1(a)\wedge\cdots\wedge g_n(a)\neq o\text.
\end{equation}
Suppose, for contradiction, that there is a least $n$ such that \eqref{vkvnX} fails. 
By the definition of $K$, $2\leq n$. Let $b=a\wedge g_1(a)\wedge\cdots\wedge g_{n-1}(a)$. By the minimality of $n$, we have that $b\neq o$ but $b\wedge g_n(a)=o$. Since $b\leq a$, we obtain that
$b\wedge g_n(b)\leq b\wedge g_n(a)$, that is, $b\wedge g_n(b)=o$. 
By Theorem~\ref{t:eq}(B), there is a $\alg\FF$-semilattice  isomorphism $\phi$ from $\alg A$ to the subalgebra $\gensub b$ such that $\phi(a)=b$. 
By Lemma~\ref{lnewie}, $\phi(o)=o$ since $\set{o}$ is the only singleton subalgebra of $\alg A$. 
Hence  $\phi\bigl(a\wedge g_n(a) \bigr) = \phi(a)\wedge g_n\bigl(\phi(a)\bigr) =b\wedge g_n(b)=o=\phi(o)$, 
which implies that 
$a\wedge g_n(a)=o$. Thus $g_n\notin K$, which is a contradiction that proves  \eqref{vkvnX}.

Next, assume that  $f_1K=\cdots =f_nK$. Then $g_2=f_2f_1^{-1},\dots, g_n=f_nf_1^{-1}$ belong to $K$, and \eqref{vkvnX} yields that $ a\wedge g_2(a)\wedge\cdots\wedge g_n(a)\neq o$. Thus
$o=f_1(o)\neq f_1(a)\wedge f_1g_2(a)\wedge \cdots\wedge f_1g_n(a)= f_1(a)\wedge\cdots\wedge f_n(a)$. This proves 
\eqref{l:kdEtb}.
 
It is obvious that $\alg U$ is a $\alg K$-semilattice, and it is generated by $a$. Hence, to give  a proof for \eqref{l:kdEtc} by  contradiction, suppose that $\alg U$ is nontrivial but it has a zero element $b$, distinct from $a$. It follows from Lemma~\ref{l:f1}\eqref{l:f1a} that $b$ is of the form $f_1(a)\wedge\cdots\wedge f_n(a)$ for some $f_1,\ldots,f_n\in K$. Since  $f_1K=\cdots=f_nK=K$,  we conclude from \eqref{l:kdEtb} that $b\neq 0$. Applying Lemma~\ref{l:f1}\eqref{l:f1c} to the $\alg K$-semilattice $\alg U$, we obtain that
\begin{equation}\label{dcCShS}
\text{for all $f\in K$, $f(b)=b$.}
\end{equation}
By Lemma~\ref{lnewie}, the $\alg\FF$-subsemilattice $\alg B=\genpsub{\alg A} b$ is not a singleton. Thus Theorem~\ref{t:eq}(B)\eqref{t:eqb} yields a $\alg\FF$-semilattice isomorphism from $\alg A$ onto $\alg B$ such that $a\mapsto b$. 
Applying the inverse of this isomorphism to \eqref{dcCShS}, we conclude that $f(a)=a$, for all $f\in K$. Thus $\alg U$, the $\alg K$-semilattice generated by $a$, is a singleton. This contradiction proves \eqref{l:kdEtc}.

Next, we deal with \eqref{l:kdEtd}. Consider an arbitrary $b\in U$. By Theorem~\ref{t:eq}(B), it suffices to show that the $\alg K$-semilattice $\alg B=\genpsub{\alg U} b$  is isomorphic to $\alg U$. We know from 
Theorem~\ref{t:eq}(B) that there is a $\alg \FF$-semilattice isomorphism $\phi\colon \alg A=\genpsub{\alg A}a\to 
\genpsub{\alg A}b$ such that $\phi(a)=b$.
This implies that, for all unary $\alg K$-semilattice terms $r$ and $s$, 
$r(a)=s(a)$ if{f} $r(b)=s(b)$. Hence we conclude that the restriction 
$\restrict\phi{\alg U}$ of $\phi$ to 
$\alg U$ is a $\alg K$-semilattice isomorphism  from $\alg U=\genpsub{\alg U}a$ to $\alg B= \genpsub{\alg U}b$. This proves~\eqref{l:kdEtd}. 

To prove \eqref{l:kdEte}, 
 fix a set $T$ of representatives of the (left) cosets of $\alg K$ such that $T\cap K=\set 1$. Consider 
$\gtwm K{\FF}U$, defined in \eqref{nfCPlTwm}, and the map
\[
\phi \colon \gtwm K{\FF}U\to \alg A,\quad\text{defined by }\pair ut\mapsto t(u)\text{ and } o\mapsto o\text.
\]
We claim that $\phi$ is an isomorphism. To see this, let $u,u_1,u_2\in U$, $t,t_1,t_2\in T$ with $t_1\neq t_2$, and $g\in \FF$. Then
\begin{align*}
\phi(\pair{u_1}{h} \wedge \pair{u_2}{h} )  &= \phi(\pair{u_1\wedge u_2}{h} )=
h(u_1\wedge u_2)= h(u_1)\wedge  h(u_2)\cr
&= \phi(\pair{u_1}{h}) \wedge\phi(\pair{u_2}{h})\text. 
\end{align*}
By Lemma~\ref{l:f1}\eqref{l:f1a}, there are $f_1,\dots,f_n,g_1,\dots,g_m\in K$ such that 
$u_1=f_1(a)\wedge\cdots\wedge f_n(a)$ and $u_2=g_1(a)\wedge\dots\wedge g_m(a)$.
Since $t_1K\neq t_2K$, the group elements $t_1f_1,\dots,t_1f_n$ and $t_2g_1,\dots,t_2g_m$ do not belong to the same coset of $\alg K$. Hence, it follows from \eqref{l:kdEtb} that $t_1(u_1)\wedge t_2(u_2)=o$. Consequently, 
\[\phi(\pair {u_1}{t_1} \wedge \pair {u_2}{t_2}) = \phi(o) =  t_1(u_1)\wedge t_2(u_2)  =\phi(\pair{u_1}{t_1})\wedge \phi(\pair{u_2}{t_2})\text.\] 
Thus $\phi$ preserves the meet operation. If $T\cap gtK=\set f$, then 
\[ \phi(g\pair ut ) = \phi( \pair{gtf^{-1}(u)}{f})= f(gtf^{-1}(u))=gt(u)=g\phi(\pair ut)\text. 
\]
Therefore, $\phi$ is a homomorphism.
Clearly, $o$ is a preimage of $o$. To show that $\phi$ is surjective, let $b\in A\setminus\set o$. By Lemma~\ref{l:f1}\eqref{l:f1a}, $b=f_1(a)\wedge\cdots\wedge f_n(a)$ for some $f_1,\ldots,f_n\in \FF$. These $f_i$ belong to the same coset $hK$ by \eqref{l:kdEtb}. Then $h^{-1}f_1,\dots,h^{-1}f_n\in K$, 
\begin{align*}
\pair{ h^{-1}&f_1(a)\wedge\dots\wedge h^{-1}f_n(a)}{h}  \text{ belongs to } \gtwm K{\FF}U,\,\text{ and }\cr
&\phi\bigl(\pair{h^{-1}f_1(a)\wedge\dots\wedge h^{-1}f_n(a)}{h}     \bigr) =  f_1(a)\wedge\dots\wedge f_n(a)=b\text.
\end{align*}
Hence, $\phi$ is surjective. 

Next, we assert that $o$, the zero element of $\alg A$, does not belong to $U$.  To obtain a contradiction, suppose that $o\in U$.  Then $|U|=1$ since otherwise $\alg U$ cannot have a zero element by    
\eqref{l:kdEtc}. Thus $|U|=1$, $U=\set a=\set o$, which implies that $|A|=1$, a contradiction. Hence, $o\notin U$.

We are now in the position to  show that $\phi$ is injective. If $\pair ut\in U\times T$, then $\phi(\pair ut)=t(u)\neq o$ since otherwise $o=t^{-1}(o)=u$ would belong to $U$. Hence the only preimage of $o$ is itself. Assume that $\pair{u_1}{t_1},\pair{u_2}{t_2}\in U\times T$ such that $\phi(\pair{u_1}{t_1})=\phi(\pair{u_2}{t_2})$. Then $t_1(u_1)=t_2(u_2)$, and 
\[o\neq \phi(\pair{u_1}{t_1})  = \phi(\pair{u_1}{t_1}) 
\wedge \phi(\pair{u_2}{t_2})=  
\phi(\pair{u_1}{t_1}\wedge  \pair{u_2}{t_2} )\text. 
\] 
Hence, $\pair{u_1}{t_1}\wedge  \pair{u_2}{t_2}\neq o$ implies that $t_1=t_2$, whence $t_1(u_1)=t_2(u_2)$ entails that $u_1=u_2$. Thus $\pair{u_1}{t_1}=\pair{u_2}{t_2}$, proving that $\phi$ is injective. 
\end{proof}

Next, we define an auxiliary set $\auxset$ as follows. Here $\Sub{\alg\FF}$ is the set of all subgroups of~$\alg\FF$, $\set 1$ is the 1-element subgroup, and $\alg{ o}=\langle\set o;\wedge,\FF\rangle$ is the one-element semilattice over $\alg\FF$.
\begin{equation}
\begin{aligned}
\auxset &=\bigl\{ \pair {K} {\alg o}:
K\in \Sub{\alg\FF}\setminus\set{\FF} 
\bigr\}\, \cup \cr
& \bigl\{ \pair{K}{\freealg{\var S}1}:
K\in \Sub{\alg\FF}\setminus\set{\set1,\FF},\,\,
\var S\in \minqvak{\alg K}\bigr\}
\text.
\end{aligned}
\end{equation}
(Of course, $\freealg{\var S}1$ above and in similar situations is understood as the isomorphism class of  $\freealg{\var S}1$. However,  we   simply speak of free algebras rather than their isomorphism classes.)
Now, we are in the position to formulate the main result of this section.

\begin{theorem}\label{rdThemt} Let $\alg\FF$ be an abelian group. Define a map
\begin{align*}
\gamma\colon \minqvan{\alg\FF}\to \auxset \,\,\text{ by }\,\,
\var R\mapsto\pair{ K_{\var R}}{\alg U_{\var R}},
\end{align*}
where $K_{\var R}$ and $\alg U_{\var R}$ are given before Lemma~\ref{l:kdEt}, and a map
\begin{align*}
\delta\colon \auxset \to  \minqvan{\alg\FF}
\,\,\text{ by }\,\, 
\pair K{\alg U}\mapsto \var Q\bigl(\gtwm K{\FF}U\bigr)\text.
\end{align*}
Then $\gamma$ and $\delta$ are reciprocal bijections. 
\end{theorem}

\begin{remark}\label{remdifD}
Theorem~\ref{rdThemt} reduces the difficulty to $\auxset$. That is, if we could describe $\minqvak{\alg K}$ for all nontrivial subgroups $K\in\Sub{\alg \FF}$, including $K=\FF$, then we would obtain a full description of $\minqvar{\alg K}$. Generally, this seems to be hopeless in view of Section~\ref{sectiExmpl}. 
\end{remark}

\begin{remark} If $\alg\FF$ is finite, then $\minqvak{\alg\FF}=\emptyset$ by Lemma~\ref{l:f1}\eqref{l:f1d}. Thus $\alg U_{\var R}=\alg o$, $\gtwm K{\FF}o\cong \qalgeb{K}{\alg\FF}$, and Theorem~\ref{rdThemt} reduces to Theorem~\ref{fIndSrpt}. That is, Theorem~\ref{rdThemt}, together with  Lemma~\ref{l:f1}\eqref{l:f1d}, implies Theorem~\ref{fIndSrpt}. However, the easy proof of Theorem~\ref{fIndSrpt} in Section~\ref{seCfinigrp} is justified by the fact that the proofs in Section~\ref{secNonFini} are much more complicated.
\end{remark}

\begin{proof}[Proof of Theorem~\ref{rdThemt}]
First we show that $\gamma$ is a map from $\minqvan{\alg\FF}$ to $\auxset$.
Let $\var R\in \minqvan{\alg\FF}$, $K=K_{\var R}$ and $\alg U= \alg U_{\var R}$.
By Lemma~\ref{l:kdEt}, it suffices to show that $K\neq\FF$ and that $|U|\neq 1$ implies $K\neq\set1$. Striving for contradiction, suppose that $K=\FF$. Since $o$ belongs to $\freealg{\var R}a$, Lemma~\ref{l:f1}\eqref{l:f1a} implies the existence of $f_1,\ldots,f_n\in\FF$ such that $o=f_1(a)\wedge\cdots\wedge f_n(a)$. But this contradicts Lemma~\ref{l:kdEt}\eqref{l:kdEtb} since any two cosets of $K=G$ are equal. Thus $K\neq \FF$. Clearly, if $K=\set 1$, then every unary $\alg K$-semilattice term induces the identity map and $|U|=1$. Hence, $|U|\neq 1$ implies that $K\neq\set 1$. This proves that $\gamma$ is a map from $\minqvan{\alg\FF}$ to $\auxset$.

Next, we show that $\delta$ is a $\auxset\to \minqvan{\alg\FF}$ map. It follows easily from Proposition~\ref{soZRc}(A)  and Remark~\ref{remfanshaped}  that
$\delta(\pair{K}{\alg o})= \qalgeb K{\alg\FF}$ belongs to $\minqvan{\alg\FF}$, provided $\FF\neq K\in\Sub{\alg \FF}$. 
So let $\pair K{\alg U}:=\pair{K}{\freealg{\var S}a}$ be in $\auxset$, where 
$K$ belongs to 
$\Sub{\alg\FF}\setminus\set{\set1,\FF}$ and $
\var S$ belongs to $\minqvak{\alg K}$. 
We have to show that $\delta(\pair K{\alg U})= \var Q\bigl(\gtwm K{\FF}U\bigr)$ belongs to $\minqvan{\alg\FF}$. 
As in \eqref{nfCPlTwm}, let $T$ be a set of representatives of the cosets of $K$ such that $1\in T$.
We know from Lemma~\ref{MdHlW} that, for every $t\in T$, 
\begin{equation}\label{sKvMYx}
\text{$\alg A=\gtwm K{\FF}U$ is generated by $\pair at$.}
\end{equation}
By Theorem~\ref{t:eq},  it suffices to prove that 
for each $\pair bt$ of $\alg A$,
the subalgebra  
$\alg B=\genpsub{\alg A}{\pair bt}$ is isomorphic to $\alg A$. 

Define a map $\phi\colon\alg A\to \alg B$ by
$o\mapsto o$ and 
$\pair {r(a)}t\mapsto \pair {r(b)}t$, where $r$ denotes an arbitrary unary $\alg K$-semilattice term. Since $\alg U= \freealg{\var S}a$ has no zero, it follows from Theorem~\ref{t:eq}(B) that there exists 
a $\alg K$-semilattice  isomorphism $\psi$ from $\alg U=\genpsub{\alg U}a$ onto $\genpsub{\alg U}b$ with $\psi(a)=b$. This implies that $\phi$ is a well-defined map  and  a bijection. 
By \eqref{sKvMYx}, each element of $\alg A$ is of the form $s(\pair at)$ for an appropriate unary $\alg \FF$-semilattice term $s$. Clearly, each element of $\alg B$ is of the form $s(\pair bt)$.
Observe that, in \eqref{nfCPlTwm}, $gtf^{-1}\in K$. Thus it is clear from the definition of $\gtwm K{\FF}U$ that the action of $s$ on a pair $\pair xt$ depends 
on two ingredients. First, it depends on how we compute within $\alg U$; from this aspect,  $\psi$ allows us to replace $a$ with $b$. Second, on how we compute with group elements; then $x$ is irrelevant. Consequently, $\phi$ is an isomorphism. 
Therefore, $\alg A=\gtwm K{\FF}U$ satisfies condition \eqref{t:eqb} of Theorem~\ref{t:eq}. Since $\delta(\pair K{\alg U})=\var Q(\alg A)$, \eqref{t:eqa}, \eqref{t:eqb}, and \eqref{t:eqe} of
Theorem~\ref{t:eq} yield that there is an isomorphism 
\begin{equation}\label{lsTZmc}
\text{$\gtwm K{\FF}U\to  \freealg{\delta(\pair K{\alg U})}{d}$ with $\pair a1\mapsto d$, }
\end{equation}
and that $\delta(\pair K{\alg U})\in \minqvan{\alg\FF}$. This proves that $\delta$ is a map from $\auxset$ to $\minqvan{\alg\FF}$.

Let $\var R\in \minqvan{\alg\FF}$. 
Since $\gtwm {K_{\var R}}{\FF}U_{\var R}\cong \freealg{\var R}1$ by Lemma~\ref{l:kdEt}\eqref{l:kdEte} and  $\var R$ is generated by any of its nontrivial algebra, we obtain that 
\[\delta\bigl(\gamma(\var R)\bigr)=\delta(\pair{K_{\var R}}{\alg U_{\var R}}) = \var Q\bigl(\gtwm {K_{\var R}}{\FF}U_{\var R}\bigr)= \var Q\bigl(\freealg{\var R}1 \bigr) =\var R\text.
\]
That is,  $\delta\circ\gamma$ is the identity map on $\minqvan{\alg\FF}$.

Next, to show that $\gamma\circ\delta$ is the identity map on $\auxset$, assume that $\pair K{\alg U}$ belongs to $ \auxset$. We distinguish two cases. 

First, assume that $\alg U$ is nontrivial. Let $\var R=\delta(\pair K{\alg U})=\var Q\bigl(\gtwm K{\FF}U\bigr)$. Then,  
by  \eqref{lsTZmc}, we can compute $\gamma\bigl(\delta(\pair K{\alg U})\bigr)= \gamma(\var R) =\pair{K_{\var R}}{\alg U_{\var R}}$ based on the algebra $\gtwm K{\FF}U$ and its free generator $\pair a1$. Hence $K_{\var R}=\{g\in \FF: \pair a1\wedge g(\pair a1)\neq o\}$. Since $o\notin U$, it is clear from definitions that $K_{\var R}=K$. It is also clear that the $\alg K$-semilattice  generated by $\pair a1$ is $U\times\set 1$, which is isomorphic to $\alg U$. Therefore,
since now isomorphic algebras are treated as equal ones, $\gamma\bigl(\delta(\pair K{\alg U})\bigr)=\pair{K}{\alg U}$.

Second, assume that $\alg U=\alg o$. 
Let $\var R=\delta(\pair K{\alg o})$.
By definitions, Remark~\ref{remfanshaped}, and Proposition~\ref{soZRc}(A), we obtain that 
\[\var R= Q\bigl(\gtwm {K}{\FF}o\bigr) = 
 Q\bigl(\qalgeb{K}{\FF}\bigr)\text{ and } \qalgeb{K}{\FF}\cong \freealg{\var R}1 \text.
\]
Hence, instead of $\freealg{\var R}1$, we can compute  $\gamma(\var R)=\pair{K_{\var R}}{\alg U_{\var R}}$ from  $\qalgeb{K}{\FF}$ and its generating element $K=1K$. Hence, clearly, we obtain that $\gamma(\var R)=\pair{K}{\alg o}$, that is, 
$\gamma\bigl(\delta(\pair K{\alg U})\bigr)=\pair K{\alg U}$. 
Thus $\gamma\circ\delta$ is the identity map.
\end{proof}

\section{An example}\label{sectiExmpl}
Since $\minqvak{\algcikl \infty}$ is less complicated than  $\minqvan{\algcikl \infty}$, see at the beginning of Section~\ref{secNonFini}, one might hope that $\minqvak{\alg\FF}$ can somehow be described for any abelian group $\alg\FF$. This hope is minimized by the following example.

Consider $\algcikl \infty^2$, the direct square of the infinite cyclic group. 
It is generated, in fact freely generated, by $\set{(g,1),(1,g)}$. It has only countably many subgroups by, say, 
 W.\,R.~Scott~\cite[Theorem 5.3.5]{Scott}. Hence, if we had that $|\minqvak{\algcikl \infty}|\leq \aleph_0$, then 
$|\minqvar{\algcikl \infty^2}| = \aleph_0$ would follow from 
Theorem~\ref{rdThemt}, and we could expect a reasonable description of  $\minqvar{\algcikl \infty^2}$. However, we construct continuously many members of $\minqvak{\algcikl \infty^2}$.

Given an irrational number $\alpha$, let $B_\alpha=\set{m+n\alpha: m,n\in\mathbb Z}$. Define the action of 
$\pair{g^i}{g^j}$  by $\pair{g^i}{g^j}(m+n\alpha)=m+i+(n+j)\alpha$. This way we obtain a $\algcikl \infty^2$-semilattice 
$\alg B_\alpha=\langle B_\alpha;\wedge,\ciklhat{\infty}2 \rangle$, where $\wedge$ is the minimum with respect to the usual order of real numbers. 

\begin{example} For each irrational number $\alpha$, $\var Q(\alg B_\alpha)\in\minqvak{\algcikl \infty^2}$.
If $\alpha$ and $\beta$ are distinct irrational numbers, then $\var Q(\alg B_\alpha)\neq \var Q(\alg B_\beta)$.
\end{example}

\begin{proof} The first part follows from 
Theorem~\ref{t:eq} since $\alg B_\alpha$ is generated by each of its elements. To prove the second part, assume that $\alpha<\beta$. We can pick $p,q\in \mathbb Z$ such that $\alpha<p/q < \beta$. Since $q\alpha < p < q\beta$, the identity 
\[\pair{g^p}1(x) \wedge \pair 1{g^q}(x) \approx \pair 1{g^q}(x) 
\]
holds in $\var Q(\alg B_\alpha)$ but fails in $\var Q(\alg B_\beta)$.
\end{proof}

%

\begin{ackno} The author is deeply indebted to Dr.\  Mikl\'os Mar\'oti  for an excellent  introduction to Universal Algebra,   for raising the problem the present paper deals with, and for his valuable friendly support since then. Also, the help obtained from Dr.\  \'Agnes Szendrei is gratefully acknowledged. 
\end{ackno}

\end{document}